\documentclass[12pt,a4paper]{amsart}
\usepackage[latin1]{inputenc}
\usepackage{graphicx}
\usepackage[all]{xy}
\CompileMatrices
\usepackage{amssymb,amsmath}
\usepackage{geometry}
\usepackage{enumerate}
\input xy
\xyoption{all}

\newtheorem{theorem}{Theorem}[section]
\newtheorem{lemma}[theorem]{Lemma}
\newtheorem{corollary}[theorem]{Corollary}
\newtheorem{proposition}[theorem]{Proposition}

\theoremstyle{definition}
\newtheorem{definition}[theorem]{Definition}

\newtheorem{remark}[theorem]{Remark}

\newcommand\remove[1]{}

\newcommand{\rnote}[1]{}
\newcommand{\jnote}[1]{}

\newcommand{\e}{\varepsilon}

\begin{document}
\title[Strictly singular multiplication operators]{\hbox{Strictly singular multiplication operators on $\mathcal L(X)$}}

\author{Martin Mathieu}
\address{Mathematical Sciences Research Centre, Queen's University Belfast, University Road, Belfast BT7 1NN, Northern Ireland}
\email{m.m@qub.ac.uk}

\author{Pedro Tradacete}
\address{Instituto de Ciencias Matem\'aticas (CSIC-UAM-UC3M-UCM),
Consejo Superior de Investigaciones Cient\'ificas,
C/ Nicol\'as Cabrera, 13--15, Campus de Cantoblanco UAM, E-28049 Madrid, Spain.}
\email{pedro.tradacete@icmat.es}

\thanks{The second-named author acknowledges financial support from the Spanish Ministry of Economy and Competitiveness through grants
MTM2016-75196-P, MTM2016-76808-P, and the ``Severo Ochoa Programme for Centres of Excellence in R\&D'' (SEV-2015-0554).
Support of an LMS Research in Pairs grant is also gratefully acknowledged.
We thank Prof.\ W.\ B.\ Johnson for helpful discussions concerning this paper. We are grateful to the referee for carefully reading the paper and pointing out an oversight in a first version of the proof of Theorem 5.5}



\begin{abstract}
Exploiting several $\ell_p$-factorization results for strictly singular operators,
we study the strict singularity of the multiplication operator $L_A R_B\colon T\mapsto ATB$ on
$\mathcal L(X)$ for various Banach spaces~$X$.
\end{abstract}

\maketitle

\section{Introduction}

\noindent
Let $X$ be a Banach space and let $\mathcal L(X)$ denote the Banach algebra of bounded linear operators from $X$ to itself.
Given $A,B\in\mathcal L(X)$ let us consider the multiplication operator $L_A R_B\colon\mathcal L(X)\rightarrow\mathcal L(X)$ given by
$L_AR_B(T)=ATB$. Properties of $L_AR_B$, and more general elementary operators, that is, finite sums of multiplication operators,
have been studied for many years from a variety of viewpoints. We mention spectral theory, Fredholm theory, compactness properties,
norms, positivity and numerous others, depending on the nature of the Banach space~$X$.
Various survey articles are contained in the proceedings volumes~\cite{CuMM} and~\cite{MM92},
and the paper~\cite{ST:survey} is especially pertinent here.

The aim of this note is to explore the strict singularity of the operator $L_AR_B$. Recall that an operator is strictly singular
if it is not an isomorphism when restricted to any infinite-dimensional subspace of its domain. In other words,
$T\in\mathcal L(X,Y)$ is strictly singular if, for every infinite-dimensional subspace $X_0\subseteq X$ and every $\varepsilon>0$,
there is $x\in X_0$ such that $\|Tx\|_Y\leq\varepsilon\|x\|_X$.
The class of strictly singular operators forms a closed two-sided operator ideal which contains the compact operators.
These operators were introduced by T. Kato~\cite{Kato} in connection with the perturbation theory of Fredholm operators;
in particular, it is well known that the spectrum of a strictly singular operator has the same structure as that of a compact operator.

In most cases, the class of strictly singular operators is strictly larger than that of compact operators. The formal inclusion
$i\colon\ell_p\hookrightarrow\ell_q$  for $1<p<q<\infty$ provides a simple example of a strictly singular operator which is not compact.
Nevertheless, Pitt's theorem (\cite[Theorem~2.1.4]{AK}) asserts that for $p<q$, every operator $T\colon\ell_q\rightarrow\ell_p$ is compact.
As a consequence, one can also deduce that on $\mathcal L(\ell_p)$ the classes of compact and strictly singular operators
coincide (cf.~\cite[p.~76]{LT}).

Our approach in this paper mainly focuses on analyzing strict singularity via factorization through certain spaces.
In particular, we will show that for every $p<q$ and $A,B\in\mathcal L(\ell_p,\ell_q)$, the operator
$L_A R_B\colon\mathcal L(\ell_q,\ell_p)\rightarrow \mathcal L(\ell_p,\ell_q)$ is strictly singular (Theorem~\ref{t:factor_pq}).
Motivated by this fact, we introduce the notion of \emph{approximately $(\ell_p,\ell_q)$-factorizable operator}, which yields a formally stronger condition than that of strictly singular operator. We will show that if $A,B$ are approximately $(\ell_p,\ell_q)$-factorizable operators, then the corresponding multiplication operator $L_AR_B$ is strictly singular (Theorem~\ref{t:factorizable}).

It is easy to check that if $L_AR_B$ is strictly singular on $\mathcal L(X)$, then so are the operators $A$ and $B^*$ (see Section~\ref{s:preliminaries}). The converse implication has been analyzed by M. Lindstr\"om, E. Saksman, and H.O. Tylli in~\cite{LST}, where it was shown to hold for
$X$ being any of the following classical Banach spaces: $L_p[0,1]$ ($1<p<\infty$); $\ell_p\oplus \ell_q$ ($1<p\leq q<\infty$); $C(K)$
(for compact Hausdorff~$K$); and $\mathcal L^1$.

Recall that strictly singular operators on $C(K)$ spaces, and more generally, on $C^*$-algebras, are weakly compact
(see~\cite{OS,Pelczynski}). We will show that, if $X$ is a reflexive Banach space with an unconditional basis, then
every multiplication operator $L_AR_B$ on $\mathcal L(X)$ which is strictly singular is necessarily
weakly compact (Corollary~\ref{c:SS-WK}).
This is also related to the fact that the composition of two strictly singular operators on certain spaces yields a compact operator.

Finally, in the last part of the paper, we provide a factorization property of strictly singular operators on $L_p$, based on the classical interpolation construction due to W. J. Davis, T. Figiel, W. B. Johnson and A. Pe{\l}czy\'{n}ski in~\cite{DFJP} and another factorization result of W. B. Johnson
given in~\cite{J}. Namely, every such operator factors through a certain Banach lattice, sufficiently separated from $L_p$, and through $\ell_p$
(see Theorem~\ref{t:factorizationSS} for the precise statement).
As a consequence of this fact, we show that when $A,B$ are strictly singular on $L_p$, the multiplication operator $L_AR_B$ factors through
the space of compact operators on~$\ell_p$ (Theorem~\ref{t:factorKellp}).

\section{Preliminaries}\label{s:preliminaries}

\noindent
Recall that an operator is compact when it maps the unit ball into a relatively compact set. Since the unit ball of an infinite-dimensional Banach space is never compact, it follows that compact operators are in particular strictly singular. Whereas the former is a purely topological notion, the latter is really about infinite-dimensional structure. Moreover, it is well known that an operator $T\colon X\rightarrow Y$ is strictly singular
if and only if, for each infinite-dimensional subspace $X_0\subseteq X$, there is a further infinite-dimensional $X_1\subseteq X_0$ such that the restriction $T|_{X_1}$ is compact (cf. \cite[Proposition 2.c.4]{LT}).

Occasionally, we will need certain generalizations of strictly singular operators. Given some Banach space $X$, an operator
$T\colon Y\rightarrow Z$ is called \emph{$X$-singular\/} provided it is not an isomorphism when restricted to any subspace of
$Y$ linearly isomorphic to~$X$. Particularly useful classes are that of $\ell_p$-singular or $c_0$-singular operators (see~\cite{JS}).
For instance, for an operator $T\colon L_p\rightarrow L_p$ being $\ell_2$-singular and $\ell_p$-singular is enough to get strict singularity~\cite{W}.

Given Banach spaces $X,Y_1,Y_2$ and an operator $A\colon Y_1\rightarrow Y_2$ let us consider the left multiplication operator
$$
\begin{array}{cccc}
L_{A;X}:  & \mathcal L(X,Y_1)  &\longrightarrow &\mathcal L(X,Y_2)   \\
          & T  & \longmapsto & AT
\end{array}
$$
as well as the right multiplication operator
$$
\begin{array}{cccc}
R_{A;X}:  & \mathcal L(Y_2,X)  &\longrightarrow &\mathcal L(Y_1,X)   \\
          & T  & \longmapsto & TA
\end{array}
$$
When there is no ambiguity about the space $X$, we will simply write $L_A$ and $R_A$ instead of $L_{A;X}$ and $R_{A;X}$.

Let $X_1, X_2, X_3, X_4$ be Banach spaces, $A\in\mathcal L(X_1,X_2)$ and $B\in\mathcal L(X_3,X_4)$.
We consider the multiplication operator $L_AR_B\colon\mathcal L(X_4,X_1)\rightarrow \mathcal L(X_3,X_2)$ given by $L_AR_B(T)=ATB$.

Note that, if we choose $x_1\in X_1,x_2^*\in X_2^*,x_3\in X_3,x_4^*\in X_4^*$ such that $x_2^*(Ax_1)=1=x_4^*(Bx_3)$,
then after considering the operators
$$
\begin{array}{cccccccccc}
J_{x_4^*}:&X_1&\longrightarrow &\mathcal L(X_4,X_1)  &  & &J_{x_1}:&X_4^*&\longrightarrow &\mathcal L(X_4,X_1)   \\
  & x  & \longmapsto & x_4^*\otimes x & & & & x^* &\longmapsto & x^*\otimes x_1
\end{array}
$$

$$
\begin{array}{cccccccccc}
\delta_{x_3}:&\mathcal L(X_3,X_2)&\longrightarrow &X_2& & &\delta_{x_2^*}:&\mathcal L(X_3,X_2)&\longrightarrow &X_3^*\\
& T&\longmapsto & Tx_3 & & & & T &\longmapsto & T^*x_2^*
\end{array}
$$
we have the  following commutative diagrams
$$\xymatrix{X_1\ar_{J_{x_4^*}}[d]\ar[rr]^A&&X_2&&X_4^*\ar_{J_{x_1}}[d]\ar[rr]^{B^*}&&X_3^*\\
\mathcal L(X_4,X_1)\ar[rr]^{L_AR_B}&&\mathcal L(X_3,X_2) \ar_{\delta_{x_3}}[u]&&\mathcal L(X_4,X_1)\ar[rr]^{L_AR_B}&&\mathcal L(X_3,X_2) \ar_{\delta_{x_2^*}}[u] }$$

\smallskip\noindent
This shows that $A$ and $B^*$ belong to the ideal generated by $L_AR_B$. In particular, if $L_AR_B$ is strictly singular,
then so are $A$ and $B^*$. Note that in general, the class of strictly singular operators is not closed under taking adjoints.

In the following, $\mathcal S(X,Y)$ and $\mathcal K(X,Y)$ will denote the spaces of strictly singular and of compact
operators, respectively, between the Banach spaces $X$ and~$Y$.

\section{Strict singularity and compactness}

\noindent
In this section we will study the relation between strict singularity and compactness of the multiplication operator.

When one of the operator coefficients is compact, the other strictly singular, and the space $X$ has the approximation property, which allows us to approximate compact operators by finite rank ones, then the multiplication operator is strictly singular. A version of this fact for more general operator ideals can be found in~\cite{LS}, but we include here a simple proof for convenience and to motivate further results.

\begin{proposition}
Let $X$, $Y$ be Banach spaces such that $Y$ has the approximation property and let $A,B\in\mathcal L(X,Y)$.
If $A\in\mathcal S(X,Y)$ and $B\in \mathcal K(X,Y)$, then $L_AR_B\colon\mathcal L(Y,X)\rightarrow\mathcal L(X,Y)$ is strictly singular.
\end{proposition}

\begin{proof}
Let us start with a weaker version of the statement.

{\bf Claim:} If $A$ is strictly singular and $B$ is a rank one operator, then $L_AR_B$ is strictly singular.

Indeed, let $y_0\in Y$ and $x_0^*\in X^*$ such that $B(x)=x_0^*(x)y_0$ for every $x\in X$. Suppose $L_AR_B$ is not strictly singular;
then there exist a normalised basic sequence $(T_n)\subseteq \mathcal L(Y,X)$ and $\alpha>0$ such that
\begin{equation}\label{eq:bdbelow}
\bigl\|\sum_n a_n AT_nB\bigr\|\geq \alpha \bigl\|\sum_n a_n T_n\bigr\|
\end{equation}
for every sequence $(a_n)$ of scalars.

Without loss of generality, we can assume that the linear span of $(T_n y_0)$ in $X$ is infinite dimensional;
indeed, otherwise pick $z^*\in Y^*$ such that $\|z^*\|=1$ and $z^*(y_0)\neq 0$ and for $\varepsilon>0$, let $(x_n)$ be an infinite sequence of linearly independent vectors in $X$ with $\|x_n\|=\varepsilon 2^{-n}$, and let
$$
\tilde {T_n}(y)=T_n(y)+z^*(y)x_n.
$$
Note that the linear span of $(\tilde {T_n}(y_0))$ is infinite dimensional and inequality~\eqref{eq:bdbelow} also holds for $\tilde {T_n}$, once $\varepsilon>0$ is small enough, as
$$
\|T_n-\tilde{T_n}\|\leq \varepsilon 2^{-n}.
$$
We have
\begin{align*}
\bigl\|A\bigl(\sum_n a_nT_ny_0\bigr)\bigr\|
    &=\bigl\|\sum_n a_n AT_ny_0\bigr\|\geq\frac{1}{\|x_0^*\|}\bigl\|\sum_n a_n AT_nB\bigr\|\\
    &\geq \frac{\alpha}{\|x_0^*\|}\bigl\|\sum_n a_n T_n\bigr\|\geq\frac{\alpha}{\|x_0^*\|\|y_0\|}\bigl\|\sum_n a_n T_ny_0\bigr\|.
\end{align*}
Hence, $A$ is bounded below on the span of $(T_nx_0)$ which is a contradiction with the fact that $A$ is strictly singular, and the claim is proved.

Now, if $B$ has finite rank, say $B=\sum_{i=1}^n B_i$ with $B_i$ of rank one, then $L_AR_B=\sum_{i=1}^n L_AR_{B_i}$ is strictly singular as a linear combination of strictly singular operators.

Finally, for a compact operator $B$, since $Y$ has the approximation property, we can find a sequence of finite rank operators $(B_n)$ with $\|B-B_n\|\rightarrow 0$. Each of $L_AR_{B_n}$ is strictly singular by the previous part of the proof, and $\|L_AR_B-L_AR_{B_n}\|\rightarrow 0$, thus, $L_AR_B$ is strictly singular too.
\end{proof}

By passing to adjoints, it easily follows
\begin{corollary}
Let $X$, $Y$ be Banach spaces such that $X^*$ has the approximation property and let $A,B\in\mathcal L(X,Y)$.
If $A\in\mathcal K(X,Y)$ and $B^*\in\mathcal S(Y^*,X^*)$, then $L_AR_B\colon\mathcal L(X,Y)\rightarrow\mathcal L(X,Y)$ is strictly singular.
\end{corollary}

It was proved in~\cite{LST} that, when $X$ is a space of the form $L_p$ for $1\leq p\leq \infty$, $C(K)$ for $K$ compact Hausdorff, or $\mathcal L^1$, $L_AR_B$ is strictly singular in $\mathcal L(X)$ if and only if so are $A$ and $B^*$. It should be noted that the composition of two strictly singular endomorphisms on $X$ for every space in the previous list yields a compact operator. On the other hand, for the spaces
$\ell_p\oplus\ell_q\oplus\ell_r$ with $1<p<q<r<\infty$ and on $L_p[0,1]\oplus L_q[0,1]$ with $1<p<q<\infty$,
$p\neq 2\neq q$, there are examples of strictly singular operators $A$ and $B^*$ such that $L_AR_B$ is not strictly singular.
And in fact, these examples are made out of strictly singular operator whose composition is not compact.
The following observation together with the results of the next section provide a reason for this.
\begin{proposition}
For a Banach space $X$ the following statements are equivalent:
\begin{enumerate}[\rm(a)]
\item For every $A,B\in\mathcal L(X)$ with $A$ and $B^*$ strictly singular, it follows that $AB$ is compact.
\item For every $A,B\in\mathcal L(X)$ with $A$ and $B^*$ strictly singular, the operator $L_AR_B$ maps
      $\mathcal L(X)$ into $\mathcal K(X)$.
\end{enumerate}
\end{proposition}
\begin{proof}
(a)${}\Rightarrow{}$(b): Let $A,B\in\mathcal L(X)$ such that $A$ and $B^*$ are strictly singular. Since strictly singular operators form an ideal in $\mathcal L(X)$, for every $T\in\mathcal L(X)$ we have that $ATB$ is compact. Hence, $L_AR_B$ maps
      $\mathcal L(X)$ into $\mathcal K(X)$.

(b)${}\Rightarrow{}$(a): Suppose $L_AR_B(T)\in\mathcal K(X)$ for $A,B,T\in\mathcal L(X)$, with  $A$ and $B^*$ strictly singular. In particular, we have that $AB=L_AR_B(I_X)\in\mathcal K(X)$.
\end{proof}

The fact that on $L_p$ spaces the composition of two strictly singular operators yields a compact operator is due to V. D. Milman~\cite{M},
and has recently been extended to further classes of Banach spaces (see~\cite{FHST}); these include for instance the Lorentz spaces
$\Lambda(w,q)$ and certain Orlicz spaces. It is conceivable that the results in~\cite{LST} extend to these larger classes of spaces.

The condition that $B^*$ above, and not $B$, should be strictly singular is clarified by the following example.
Let $X=\ell_1\oplus c_0$. Using the fact that strictly singular and compact operators coincide on $\ell_1$ and $c_0$, and that
$\mathcal L(c_0,\ell_1)=\mathcal K(c_0,\ell_1)$ it is not hard to check that, $AB\in\mathcal K(X)$ whenever $A,B\in\mathcal S(X)$.
However, let $B(x,y)=(0,qx)$ where $q\in \mathcal L(\ell_1,c_0)$ is a quotient operator.
It follows that $L_AR_B$ cannot be strictly singular (because $B^*$ is not strictly singular).

\section{Strictly singular multiplication is weakly compact}

\noindent
It is well known that strictly singular operators on $C(K)$ spaces are weakly compact \cite{Pelczynski}. This fact can also be extended to operators on $C^*$-algebras \cite[Proposition 3.1]{OS}, and we will see this is also the case for multiplication operators on $\mathcal L(X)$, for a large class of spaces~$X$.

Our main reference for weak compactness of multiplication operators on $\mathcal L(X)$ is \cite{ST}. In particular, by \cite[Corollary~2.4]{ST},
if $X$ is a reflexive space with the approximation property, the multiplication operator $L_AR_B$ is weakly compact if and only if
$ATB\in\mathcal K(X)$ for every $T\in\mathcal L(X)$.

\begin{proposition}\label{p:ABinK}
Let $X$ be a reflexive Banach space with unconditional basis, and $A,B\in\mathcal L(X)$.
If\/ $L_AR_B$ is $c_0$-singular, then $AB\in\mathcal K(X)$.
\end{proposition}
\begin{proof}
Suppose $AB\notin\mathcal K(X)$. Since $X$ is reflexive, we can find a weakly null sequence $(x_n)_{n\in\mathbb N}\subseteq X$ such that $\|ABx_n\|\geq\delta>0$ for every $n\in \mathbb N$. In particular, $(Bx_n)_{n\in \mathbb N}$ is also weakly null, and $\|Bx_n\|\geq \delta/\|A\|$ for each $n\in \mathbb N$. Now, by a standard perturbation argument we can assume that $(Bx_n)_{n\in\mathbb N}$ is a block sequence with respect to the unconditional basis of $X$. Hence, we can consider $(U_n)_{n\in \mathbb N}\subseteq \mathcal L(X)$, a sequence of projections onto each of the corresponding blocks, that is $U_n\perp U_m$ with
$$
U_nBx_n=Bx_n.
$$
We claim that $L_AR_B$ is an isomorphism on the subspace $[U_n]$. To see this, first note that, using the unconditionality of the basis of $X$, it is easy to check that for any sequence of scalars $(a_n)_{n\in\mathbb N}$ we have
\begin{equation}
\Big\|\sum_n a_n U_n\Big\|\approx \max_n |a_n|.
\end{equation}
Therefore, we have
\begin{equation}
\Big\|\sum_n a_n AU_nB\Big\|= \Big\|L_AR_B\big(\sum_n a_n U_n\big)\Big\|\lesssim \|L_AR_B\| \max_n|a_n|.
\end{equation}
While, on the other hand, we have
\begin{equation}
\Big\|\sum_n a_n AU_nB\Big\|\gtrsim\sup_j \Big\|\sum_n a_n AU_nBx_j\Big\|=\sup_j \Big\|a_j AU_jBx_j\Big\|\geq \delta\max_j|a_j|.
\end{equation}
Hence, $L_AR_B|_{[U_n]}$ is an isomorphism as claimed. Since $(U_n)$ is equivalent to the unit basis of $c_0$, the proof is finished.
\end{proof}
\begin{corollary}\label{c:SS-WK}
Let $X$ be a Banach space, and $A,B\in \mathcal L(X)$. Consider the following statements for $L_AR_B\colon\mathcal L(X)\rightarrow\mathcal L(X)$:
\begin{enumerate}[\rm(i)]
\item $L_AR_B$ is strictly singular.
\item $L_AR_B$ is $c_0$-singular.
\item $L_AR_B$ is weakly compact.
\end{enumerate}
Clearly, {\rm(i)}${}\Rightarrow{}${\rm(ii)}. If $X$ is reflexive with an unconditional basis,
then we have {\rm(ii)}${}\Rightarrow${\rm(iii)}.
\end{corollary}
\begin{proof}
Suppose $L_AR_B\colon\mathcal L(X)\rightarrow\mathcal L(X)$ is $c_0$-singular.
Then, for every $T\in\mathcal L(X)$, we have that $L_AR_BL_T=L_{AT}R_B$ is $c_0$-singular.
Hence, by Proposition~\ref{p:ABinK}, it follows that $ATB\in\mathcal K(X)$ for every $T\in\mathcal L(X)$.
As a result, $L_AR_B(\mathcal L(X))\subseteq\mathcal K(X)$, and \cite[Corollary 2.4]{ST} yields the claim.
\end{proof}

Note that for every infinite-dimensional reflexive space~$X$, there are weakly compact multiplication operators on $\mathcal L(X)$ which are not strictly singular. Indeed, let $A\in\mathcal K(X)$ and $B=I_X$, then by \cite[Proposition 2.8]{ST} $L_AR_B$ is weakly compact, but $L_AR_B$ is not strictly singular as $B^*$ is not strictly singular. The same would hold for non-reflexive $X$ as far as there is a weakly compact operator $B\in\mathcal L(X)$ such that $B^*$ is not strictly singular.

\begin{remark}
Proposition~\ref{p:ABinK} can be extended to the more general case when $X$ has an unconditional finite-dimensional decomposition.
\end{remark}

\section{Factorization of multiplication operators}\label{s:factorization}

\noindent
A classical result by J. Holub~\cite[Theorem 1]{Holub} states that every subspace of $\mathcal K(\ell_2)$ is either isomorphic to $\ell_2$ or contains a further subspace isomorphic to~$c_0$. We need a version of this dichotomy for $\mathcal K(\ell_p,\ell_q)$. Throughout this section we assume $1<p,q<\infty$.

First, recall that to any operator $T\in \mathcal K(\ell_p,\ell_q)$, we can associate the infinite matrix given by $T_{ij}=e_i^*(Te_j)$ for $i,j\in\mathbb N$, where $e_i^*,e_j$ denote the (unconditional) unit vector basis of $\ell_{q'}$ and $\ell_p$, respectively. For $n\in\mathbb N$, we will consider two particular projections in $\mathcal K(\ell_p,\ell_q)$, $E_n$ and $P_n$, given for $T\in\mathcal K(\ell_p,\ell_q)$, by
$$
E_n(T)_{ij}=\left\{
         \begin{array}{cc}
           T_{ij} & \textrm{ if }\min\{i,j\}< n, \\
           0 & \textrm{ otherwise.} \\
         \end{array}
       \right.
$$
$$
P_n(T)_{ij}=\left\{
         \begin{array}{cc}
           T_{ij} & \textrm{ if }\max\{i,j\}\leq n, \\
           0 & \textrm{ otherwise.} \\
         \end{array}
       \right.
$$
It is well known that these define a family of uniformly bounded projections on $\mathcal K(\ell_p,\ell_q)$. Let $C=\sup_{n\in\mathbb N}\max\{\|P_n\|,\|E_n\|\}$.

Given natural numbers $m<n$, and $1<p<\infty$, let $Q_{[m,n]}^{(p)}$ denote the basis projection onto the span of $(e_i)_{i=m}^n$ in $\ell_p$; or, in other words, for $(x_i)\in \ell_p$,
$$
Q_{[m,n]}^{(p)}\Big(\sum_{i\in\mathbb N} x_i e_i\Big)=\sum_{i=m}^n x_i e_i.
$$
Clearly, if $m_1<n_1<m_2<n_2$, then we have $Q_{[m_1,n_1]}^{(p)}Q_{[m_2,n_2]}^{(p)}=Q_{[m_2,n_2]}^{(p)}Q_{[m_1,n_1]}^{(p)}=0$.

We will say that $(S_k)\subset \mathcal K(\ell_p,\ell_q)$ is a \textit{block-diagonal sequence} when for every $k\in\mathbb N$ there exist $p_k<q_k<p_{k+1}$ so that,
$$
S_k=Q^{(q)}_{[p_k,q_k]} S_k Q^{(p)}_{[p_k,q_k]}.
$$
Also, a single operator $S\in \mathcal L (\ell_p,\ell_q)$ is called \textit{block-diagonal} if there is a block-diagonal sequence $(S_k)$ such that for every $x\in \ell_p$, $Sx=\sum_{k\in\mathbb N} S_k x$.

\begin{lemma}\label{block-diagonal}
Let $1<p, q<\infty$, and $(S_k)\subset \mathcal K(\ell_p,\ell_q)$ a semi-normalised block-diagonal sequence. If $p\leq q$, then $(S_k)$ is equivalent to the unit vector basis of $c_0$. If $q<p$, then $(S_k)$ is equivalent to the unit vector basis of $\ell_r$  with $r=\frac{pq}{p-q}$.
\end{lemma}

\begin{proof}
By hypothesis, for every $k\in\mathbb N$ there exist $p_k<q_k<p_{k+1}$ so that
$$
S_k=Q^{(q)}_{[p_k,q_k]} S_k Q^{(p)}_{[p_k,q_k]}.
$$

Suppose first that $p\leq q$, and let us see that in this case, $(S_k)$ is equivalent to the unit vector basis of~$c_0$. Indeed, given scalars $(a_k)$, since $c=\inf_k\|S_k\|>0$, for every $k\in\mathbb N$ there is $x_k\in \ell_p$ with $\|x_k\|_p=1$, $\|S_k x_k\|_q\geq c$ and $Q^{(p)}_{[p_k,q_k]}x_k=x_k$. Thus, for every $k\in\mathbb N$,
$$
\bigl\|\sum_{j\in\mathbb N} a_j S_j\bigr\|\geq\bigl\|\sum_{j\in\mathbb N} a_j S_j(x_k)\bigr\|_q=\bigl\|\sum_{j\in\mathbb N} a_j Q^{(q)}_{[p_j,q_j]} S_j Q^{(p)}_{[p_j,q_j]}(x_k)\bigr\|_q=\|a_k S_k(x_k)\|_q\geq c|a_k|,
$$
which yields the estimate
$$
\bigl\|\sum_{j\in\mathbb N} a_j S_j\bigr\|\geq c\sup_{j\in\mathbb N}|a_j|.
$$
On the other hand, let $K=\sup_k \|S_k\|$ and for $x \in\ell_p$ with $\|x\|_p=1$, let $x_k=Q^{(p)}_{[p_k,q_k]} x$. Note that
$$
S_k(x)= S_k(x_k).
$$
Hence, for scalars $(a_k)$ we have
\begin{eqnarray*}
\bigl\|\sum_{j\in\mathbb N} a_j S_j x\bigr\|_q  &=&  \Big(\sum_{j\in\mathbb N} \|a_j S_j(x)\|_q^q\Big)^{\frac1q} \\
   &=& \Big(\sum_{j\in\mathbb N} \|a_j S_j(x_j)\|_q^q\Big)^{\frac1q} \\
   &\leq& \sup_{j\in\mathbb N}|a_j|\|S_j\|\Big(\sum_{j\in\mathbb N} \|x_j\|_p^p\Big)^{\frac1p} \\
   &\leq& K\sup_{j\in\mathbb N}|a_j|.
\end{eqnarray*}
Therefore, $(S_k)$ is equivalent to the unit vector basis of~$c_0$ as claimed.

Now, suppose that $q<p$. As above, given scalars $(a_j)$, letting $K=\sup_j \|S_j\|$, we have
\begin{eqnarray*}
\bigl\|\sum_{j\in\mathbb N} a_j S_j \bigr\|  &=& \sup_{\|x\|_p\leq1} \bigl\|\sum_{j\in\mathbb N} a_j S_j x\bigr\|_q  =\sup_{\|x\|_p\leq1} \bigl\|\sum_{j\in\mathbb N} a_j Q^{(q)}_{[p_j,q_j]} S_j Q^{(p)}_{[p_j,q_j]} x\bigr\|_q\\
 &=&\sup_{\|x\|_p\leq1}\Big(\sum_{j\in\mathbb N} \|a_j Q^{(q)}_{[p_j,q_j]} S_j Q^{(p)}_{[p_j,q_j]} x\|_q^q\Big)^{\frac1q} \leq K \sup_{\|x\|_p\leq1}\Big(\sum_{j\in\mathbb N} |a_j|^q \| Q^{(p)}_{[p_j,q_j]} x\|_p^q\Big)^{\frac1q}.
\end{eqnarray*}
Now, if we set $s=\frac{p}{q}>1$ and $\frac1s+\frac1{s'}=1$, then by H\"older's inequality it follows that
$$
\bigl\|\sum_{j\in\mathbb N} a_j S_j \bigr\| \leq K \sup_{\|x\|_p\leq1} \Big(\sum_{j\in\mathbb N} |a_j|^{qs'}\Big)^{\frac1{qs'}}\Big(\sum_{j\in\mathbb N} \| Q^{(p)}_j x\|_p^{qs}\Big)^{\frac1{qs}}\leq K\Big(\sum_{j\in\mathbb N} |a_j|^r\Big)^{\frac1r},
$$
where $r=\frac{pq}{p-q}$.

For the converse inequality, as above, since $c=\inf_k\|S_k\|>0$, for every $k\in\mathbb N$ there is $x_k\in \ell_p$ with $\|x_k\|_p=1$, $\|S_k x_k\|_q\geq c$ with $Q^{(p)}_{[p_k,q_k]}x_k=x_k$. Given any sequence $(a_k)\in \ell_r$, let $x=\sum_k |a_k|^{\frac{r-q}{q}} x_k.$ Note that
$$
\|x\|_p=\Big(\sum_{k\in\mathbb N}  |a_k|^{\frac{r-q}{q}p}\Big)^{\frac1p}=\Big(\sum_{k\in\mathbb N}  |a_k|^r\Big)^{\frac1p}.
$$
Hence, we have
\begin{eqnarray*}
\bigl\|\sum_{j\in\mathbb N} a_j S_j \bigr\|  &\geq&\frac{\bigl\|\sum_{j\in\mathbb N} a_j S_j x\bigr\|_q}{\|x\|_p}=\frac{\Big(\sum_{j\in\mathbb N} \|a_j S_j x\|_q^q\Big)^{\frac1q}}{\|x\|_p}\\
&\geq&c \Big(\sum_{j\in\mathbb N} |a_j|^r\Big)^{\frac1q-\frac1p}=c\Big(\sum_{j\in\mathbb N} |a_j|^r\Big)^{\frac1r},
\end{eqnarray*}
as claimed.
\end{proof}

\begin{lemma}\label{l:p-Holub}
Let $1<p, q<\infty$, $\frac1p+\frac{1}{p'}=1$ and $M$ a closed infinite-dimensional subspace of $\mathcal K(\ell_p,\ell_q)$. The following dichotomy holds:
\begin{enumerate}
\item There is $n\in\mathbb N$ such that $E_n|_M$ is an isomorphism, in which case $M$ contains an isomorphic copy of $\ell_q$ or $\ell_{p'}$; or,
\item there exist a normalised sequence $(T_k)\subset M$ and a semi-normalised block-diagonal sequence $(S_k)$ such that $\|T_k-S_k\|\rightarrow 0$ as $k\rightarrow \infty$. In particular, $(T_k)$ is equivalent to the unit vector basis of $c_0$ when $p\leq q$, and to the unit vector basis of $\ell_r$ with $r=\frac{pq}{p-q}$, when $q<p$.
\end{enumerate}
\end{lemma}
\begin{proof}

Suppose first that there is $n\in\mathbb N$ such that the restriction $E_n|_M$ is an isomorphism. Since the range of $E_n$ is isomorphic to $\ell_q\oplus\ell_{p'}$ it follows that $M$ either contains a subspace isomorphic to $\ell_q$ or~$\ell_{p'}$.

On the other hand, suppose that for every $n\in\mathbb N$, $E_n|_M$ is never an isomorphism. In this case, we will inductively construct two sequences as required. Indeed, pick arbitrary $T_1\in M$ with $\|T_1\|=1$ and, by compactness, let $n_1\in \mathbb N$ such that
$$
\|T_1-P_{n_1}(T_1)\|<\frac12.
$$
Since, $E_{n_1}|_M$ is not an isomorphism, there is $T_2\in M$ with $\|T_2\|=1$ and
$$
\|E_{n_1}(T_2)\|<\frac12.
$$
Let $n_2\in\mathbb N$ be such that
$$
\|T_2-P_{n_2}(T_2)\|<\frac14.
$$
Continuing in this way, we produce inductively an increasing sequence $(n_k)_{k\in\mathbb N}\subseteq \mathbb N$ and $(T_k)_{k\in\mathbb N}\subseteq M$ such that for every $k\in \mathbb N$:
\begin{enumerate}
\item $\|T_k\|=1$
\item $\|E_{n_k}(T_{k+1})\|<2^{-k}$
\item $\|T_k-P_{n_k}(T_k)\|<2^{-k}$.
\end{enumerate}
Let $S_k=P_{n_k}(T_k)-E_{n_{k-1}}P_{n_k}(T_k)$, which clearly satisfy
$$
S_k=Q^{(q)}_{[n_{k-1}+1,n_k]} S_k Q^{(p)}_{[n_{k-1}+1,n_k]}.
$$
We have
$$
\|T_k-S_k\|\leq \|T_k-P_{n_k}(T_k)\|+\|E_{n_{k-1}}(T_k)\|+\|E_{n_{k-1}}(T_k-P_{n_k}(T_k))\|\leq (C+3)2^{-k}
$$
In particular, $(T_k)$ and $(S_k)$ are equivalent basic sequences in $\mathcal K(\ell_p,\ell_q)$.  The conclussion follows from Lemma \ref{block-diagonal}.
\end{proof}

\begin{remark}\label{r:sums}
The previous argument also holds for $\mathcal K((\oplus X_n)_{\ell_p},(\oplus Y_n)_{\ell_q})$,
where $(X_n)$ and $(Y_n)$ are sequences of finite-dimensional subspaces.
\end{remark}

\begin{remark}
When $1<q<p<\infty$, since $\mathcal L(\ell_p,\ell_q)=\mathcal K(\ell_p,\ell_q)$, by \cite[Corollary 2]{K}, we know that $\mathcal K(\ell_p,\ell_q)$ is a reflexive space. The above lemma is somehow more informative, since any infinite dimensional subspace of $\mathcal K(\ell_p,\ell_q)$ contains one of the reflexive spaces $\ell_q$, $\ell_{p'}$ or $\ell_r$.
\end{remark}

\begin{theorem}\label{t:factor_pq}
Let $p<q$ and $A,B\in\mathcal L(\ell_p,\ell_q)$. Then $L_AR_B\colon\mathcal L(\ell_q,\ell_p)\rightarrow\mathcal L(\ell_p,\ell_q)$ is strictly singular.
\end{theorem}
\begin{proof}
Note that by Pitt's theorem $\mathcal L(\ell_q,\ell_p)=\mathcal K(\ell_q,\ell_p)$, so we simply consider the operator $L_AR_B\colon\mathcal K(\ell_q,\ell_p)\rightarrow\mathcal K(\ell_p,\ell_q)$.

First, let us assume that $A,B$ are both block-diagonal operators. Suppose $L_{A}R_{B}$ is not strictly singular. Therefore, there exists a closed subspace $M\subseteq\mathcal K(\ell_q,\ell_p)$ such that $L_AR_B|_M$ is an isomorphism. By Lemma \ref{l:p-Holub}, either \begin{enumerate}
\item[(1)] $M$ contains a subspace isomorphic to $\ell_p$ or $\ell_{q'}$; or,
\item[(2)] there exist a normalised sequence $(T_n)\subset M$ and a semi-normalised block-diagonal sequence $(S_n)\subset \mathcal K(\ell_q,\ell_p)$ such that $\|T_n-S_n\|\rightarrow 0$.
\end{enumerate}

In case (1), by another application of Lemma \ref{l:p-Holub} to the subspace $L_AR_B(M)\subset \mathcal K(\ell_p,\ell_q)$, we know that this space contains a further subspace which is isomorphic to $\ell_q$, $\ell_{p'}$ or $c_0$. Since $p\neq q$, and $L_AR_B|_M$ is an isomorphism, the only possibility would be that $p=p'=2$ or $q=q'=2$. In any of these cases, note that we have the factorizations
$$
\xymatrix{\mathcal K(\ell_q,\ell_p)\ar_{L_{A;\ell_q}}[dr]\ar[rr]^{L_AR_B}&&\mathcal K(\ell_p,\ell_q)&\mathcal K(\ell_q,\ell_p)\ar_{R_{B;\ell_p}}[dr]\ar[rr]^{L_AR_B}&&\mathcal K(\ell_p,\ell_q)\\
&\mathcal K(\ell_q)\ar[ur]_{R_{B;\ell_q}}&& &\mathcal K(\ell_p)\ar[ur]_{L_{A;\ell_p}}& }
$$
In particular, $M$ is isomorphic to a subspace of $\mathcal K(\ell_p)$ and also to a subspace of $\mathcal K(\ell_q)$. As $p$ and $q$ cannot both be equal to 2, by Lemma \ref{l:p-Holub}, we arrive at a contradiction. Hence, in case (1), $L_AR_B|_M$ cannot be an isomorphism.

Assume now case (2) holds. As $A$ and $B$ are block-diagonal operators, there exist $p_k<q_k<p_{k+1}$ and $r_k<s_k<r_{k+1}$ such that for $x\in \ell_p$ we have
$$
Ax=\sum_{k\in\mathbb N} Q^{(q)}_{[p_k,q_k]}AQ^{(p)}_{[p_k,q_k]}x,\quad\quad Bx=\sum_{k\in\mathbb N} Q^{(q)}_{[r_k,s_k]}BQ^{(p)}_{[r_k,s_k]}x.
$$
Also, there exist $m_k<n_k$ such that
$$
S_k= Q^{(q)}_{[m_k,n_k]}S_kQ^{(p)}_{[m_k,n_k]}.
$$

For each $k\in \mathbb N$, let
$$
\tilde m_k=\min\Big\{\min\{p_j: [p_j,q_j]\cap[m_k,n_k]\neq \emptyset\},\min\{r_j: [r_j,s_j]\cap[m_k,n_k]\neq \emptyset\}\Big\}
$$
and
$$
\tilde n_k=\max\Big\{\max\{q_j: [p_j,q_j]\cap[m_k,n_k]\neq \emptyset\},\max\{s_j: [r_j,s_j]\cap[m_k,n_k]\neq \emptyset\}\Big\}.
$$

We extract a subsequence of $(S_j)$ as follows: let $j_1=1$ and for $k\geq1$, take $j_k$ large enough so that $\tilde n_{j_{k-1}}<\tilde m_{j_k}$. By construction, it follows that for every $k\in \mathbb N$
$$
AS_{j_k}B=Q^{(q)}_{[\tilde m_{j_k},\tilde n_{j_k}]}AS_{j_k}BQ^{(p)}_{[\tilde m_{j_k},\tilde n_{j_k}]}.
$$
Hence,  $(AS_{j_k}B)$ is a semi-normalised block-diagonal sequence in $\mathcal K(\ell_p,\ell_q)$, and by Lemma \ref{block-diagonal}, $(AS_{j_k}B)$ is equivalent to the unit vector basis of $c_0$. Moreover, as $\|L_AR_B(T_{j_k}-S_{j_k})\|\rightarrow0$ when $k\rightarrow \infty$, by standard perturbation we have that  $(AT_{j_k}B)$ is also equivalent to the unit vector basis of $c_0$. However, by Lemma \ref{block-diagonal}, we know that $(S_{j_k})$ and $(T_{j_k})$ are equivalent to the unit vector basis or $\ell_r$ with $r=\frac{pq}{p-q}$. This is a contradiction with the assumption that $L_AR_B|_M$ is an isomorphism.

So far, we have shown that when both $A$ and $B$ are block-diagonal operators, then $L_AR_B:\mathcal K(\ell_q,\ell_p)\rightarrow \mathcal K(\ell_p,\ell_q)$ is strictly singular. Now, for arbitrary $A,B\in\mathcal L(\ell_p,\ell_q)$, and every $\varepsilon>0$, by \cite[Lemma 4.4(i)]{SSTJT}, there exist block-diagonal operators $A^\varepsilon_1,A^\varepsilon_2,B^\varepsilon_1,B^\varepsilon_2 \in\mathcal L(\ell_p,\ell_q)$ such that
$$
\|A-(A_1^\varepsilon + A_2^\varepsilon)\|<\varepsilon,\quad\quad  \|B-(B_1^\varepsilon + B_2^\varepsilon)\|<\varepsilon.
$$
Clearly, for every $\varepsilon>0$ we have
$$
L_{ A_1^\varepsilon + A_2^\varepsilon}R_{B_1^\varepsilon + B_2^\varepsilon}=L_{A_1^\varepsilon}R_{B_1^\varepsilon}+L_{A_1^\varepsilon}R_{B_2^\varepsilon}+L_{A_2^\varepsilon}R_{B_1^\varepsilon}+L_{A_2^\varepsilon}R_{B_2^\varepsilon}.
$$
By the above part of the proof, it follows that each of the $L_{A_i^\varepsilon}R_{B_j^\varepsilon}$ is strictly singular, thus so is the sum $L_{ A_1^\varepsilon + A_2^\varepsilon}R_{B_1^\varepsilon + B_2^\varepsilon}$. Finally, since $L_{ A_1^\varepsilon + A_2^\varepsilon}R_{B_1^\varepsilon + B_2^\varepsilon}\rightarrow L_AR_B$ as $\varepsilon\rightarrow0$, we conclude $L_AR_B$ is also strictly singular.
\end{proof}

We introduce the following class of operators.
\begin{definition}
Given $p<q$, we say that $T\in\mathcal L(X)$ is \textit{approximately $(\ell_p,\ell_q)$-factorizable\/}
if, for every $\varepsilon>0$, there exist operators $T_1^\varepsilon\colon X\rightarrow\ell_{p}$,
$T_2^\varepsilon\colon\ell_{p}\rightarrow\ell_{q}$ and $T_3^\varepsilon\colon\ell_{q}\rightarrow X$ such that
$$
\|T-T_3^\varepsilon T_2^\varepsilon T_1^\varepsilon\|\leq \varepsilon.
$$
\end{definition}

Note that the class of approximately $(\ell_p,\ell_q)$-factorizable operators forms a closed operator ideal contained in that of strictly singular operators. Indeed, with the above notation, every operator $T_2^\varepsilon\colon\ell_p\rightarrow\ell_q$ is strictly singular; since the strictly singular operators are a closed operator ideal, it follows that every approximately $(\ell_p,\ell_q)$-factorizable operator is strictly singular.
\begin{theorem}\label{t:factorizable}
Suppose that $A,B\in\mathcal L(X)$ are approximately $(\ell_p,\ell_q)$-factorizable operators for some $p<q$.
Then $L_AR_B\colon\mathcal L(X)\rightarrow\mathcal L(X)$ is strictly singular.
\end{theorem}
\begin{proof}
For every $\varepsilon>0$, let $A_1^\varepsilon,B_1^\varepsilon\colon X\rightarrow \ell_{p}$,
$A_2^\varepsilon,B_2^\varepsilon\colon\ell_{p}\rightarrow\ell_{q}$ and
$A_3^\varepsilon,B_3^\varepsilon\colon\ell_{q}\rightarrow X$ be such that $\|A-A_3^\varepsilon A_2^\varepsilon A_1^\varepsilon\|\leq\varepsilon$ and $\|B-B_3^\varepsilon B_2^\varepsilon B_1^\varepsilon\|\leq\varepsilon$.
For convenience, set $A^\varepsilon=A_3^\varepsilon A_2^\varepsilon A_1^\varepsilon$ and $B^\varepsilon=B_3^\varepsilon B_2^\varepsilon B_1^\varepsilon$. The factorizations
$$
\xymatrix{X\ar_{A^\varepsilon_1}[d]\ar[rr]^{A^\varepsilon}&&X&&X\ar_{B^\varepsilon_1}[d]\ar[rr]^{B^\varepsilon}&&X\\
\ell_{p}\ar[rr]^{A^\varepsilon_2}&&\ell_{q}\ar_{A^\varepsilon_3}[u] && \ell_{p}\ar[rr]^{B^\varepsilon_2}&&\ell_{q}\ar_{B^\varepsilon_3}[u] }
$$
yield the following factorization for the corresponding multiplication operators:
$$
\xymatrix{\mathcal L(X)\ar_{L_{A^\varepsilon_1}R_{B^\varepsilon_3}}[d]\ar[rr]^{L_{A^\varepsilon} R_{B^\varepsilon}}&&\mathcal K(X)\\
\mathcal K(\ell_q,\ell_p)\ar_{L_{A^\varepsilon_2}R_{B^\varepsilon_2}}[rr]&&\mathcal K(\ell_p,\ell_q)\ar_{L_{A^\varepsilon_3}R_{B^\varepsilon_1}}[u]}
$$
Here we use implicitly Pitt's theorem which asserts $\mathcal L(\ell_q,\ell_p)=\mathcal K(\ell_q,\ell_p)$ for $p<q$.

Theorem~\ref{t:factor_pq} yields that $L_{A_2^\varepsilon}R_{B_2^\varepsilon}$ and, hence $L_{A^\varepsilon}R_{B^\varepsilon}$, is strictly singular for every $\varepsilon>0$.
Note that
\begin{equation*}
\begin{split}
\|L_AR_B-L_{A^\varepsilon}R_{B^\varepsilon}\| &\leq \|L_A-L_{A^\e}\|\|R_B\|+\|L_{A^\e}\|\|R_B-R_{B^\e}\|\\
        &\leq\|B\|\,\e+(\|A\|+\e)\,\e.
\end{split}
\end{equation*}
The conclusion follows from the fact that the strictly singular operators form a closed ideal.
\end{proof}

\section{Factorization of strictly singular operators on $L_p$}

\noindent
In this section we focus on the case $X=L_p([0,1])$, for $1<p<\infty$ and $p\neq2$, endowed with Lebesgue measure~$\mu$. For simplicity we will always write~$L_p$.
According to \cite[Theorem 2.9]{LST}, $A,B\in\mathcal S(L_p)$ if and only if $L_AR_B$ is strictly singular on $\mathcal L(L_p)$.
However, the proof of this result is considerably long, and a more concise argument would be desirable.
Our aim here is to show how factorization techniques can shed some light in this direction.

Recall that $T\in\mathcal S(L_p)$ is equivalent to $T^*\in\mathcal S(L_{p'})$, $\frac1p+\frac1{p'}=1$~\cite{W}.
Also recall that a subset $W\subseteq L_p$ is uniformly $p$-integrable when
$$
\lim_{\mu(A)\rightarrow 0}\sup_{f\in W}\|f\chi_A\|_p=0.
$$
An argument like \cite[III.C.12]{Woj}, see also \cite[Lemma 1]{JS}, yields:
\begin{lemma}\label{l:uniformly p-integrable}
Let\/ $W\subseteq L_p$ $(p\neq2)$ be a bounded convex symmetric set. The following are equivalent:
\begin{enumerate}[\rm(a)]
\item $W$ is uniformly $p$-integrable;
\item $W$ does not contain any sequence $(x_n)$ which is equivalent to the unit vector basis of\/ $\ell_p$ and spans a complemented subspace;
\item For every $\varepsilon>0$, there is $M_\varepsilon>0$ such that\/ $W\subseteq M_\varepsilon B_{L_\infty}+\varepsilon B_{L_p}$.
\end{enumerate}
\end{lemma}

Let us start with a factorization property of strictly singular operators on~$L_p$.
\begin{theorem}\label{t:factorizationSS}
Let $p<2$ and $T\in\mathcal S(L_p)$. There are a Banach lattice $X_T\subseteq L_p$ such that the unit ball of $X_T$ is
uniformly $p$-integrable in $L_p$, and operators $R\colon L_p\rightarrow\ell_p$, $S\colon \ell_p\rightarrow X_T$
making the following diagram commutative
$$
\xymatrix{L_p\ar_{R}[d]\ar[rr]^{T}&&L_p\\
\ell_p\ar[rr]^{S}&&X_T\ar_{j}[u] }
$$
where $j\colon X_T\hookrightarrow L_p$ denotes the formal inclusion.
\end{theorem}

We prepare the proof with the following lemmas.
\begin{lemma}\label{l:ellpsingular}
Let $1<p<2$ and $T\colon L_p\rightarrow L_p$ be $\ell_p$-singular. Then $T(B_{L_p})$ is a uniformly $p$-integrable set.
\end{lemma}
\begin{proof}
Set $W_0=T(B_{L_p})$. By Lemma~\ref{l:uniformly p-integrable}, it is enough to see that $W_0$ does not contain any sequence which is equivalent to the $\ell_p$ basis. Suppose the contrary, and let $(x_n)\subseteq B_{L_p}$ be such that, for some constant $C>0$ and
any sequences $(a_n)$ of scalars, we have
$$
\frac1C \left(\sum_n |a_n|^p\right)^{\frac1p}\leq \left\|\sum_n a_n Tx_n\right\|\leq C \left(\sum_n |a_n|^p\right)^{\frac1p}.
$$
Passing to a subsequence, there is no loss of generality in assuming that $(x_n)$ is weakly null, and hence unconditional.
Using unconditionality and the fact that $L_p$ has type $p$ it follows that for some constants $K,M>0$ we have
\begin{align*}
\frac1C \left(\sum_n |a_n|^p\right)^{\frac1p}&\leq \left\|\sum_n a_n Tx_n\right\|\leq\|T\|\left\|\sum_n a_n x_n\right\|\\
&\leq K\int_0^1 \left\|\sum_n a_n r_n(t) x_n\right\|dt \leq KM \left(\sum_n |a_n|^p\right)^{\frac1p}.
\end{align*}
Therefore $T$ is invertible on the span of $(x_n)$, which is equivalent to the unit vector basis of $\ell_p$.
This contradicts the assumption that $T$ is $\ell_p$-singular.
\end{proof}

For $p<q$, let $i_{q,p}:L_q\hookrightarrow L_p$ denote the formal inclusion operator.
\begin{lemma}\label{l:ell2singular}
Let $p>2$ and\/ $T\colon L_p\rightarrow L_p$ be $\ell_2$-singular. Then the operator $i_{p,2}T\colon L_p\rightarrow L_2$ is compact.
\end{lemma}
\begin{proof}
Suppose $(x_n)\subseteq L_{p}$ is a bounded sequence which, without loss of generality, can be assumed to be weakly null and normalised, and that $\liminf_n\|i_{p,2}Tx_n\|_2>0$. Hence, we can extract a subsequence such that $(i_{p,2}Tx_n)$ is equivalent to the unit basis of $\ell_2$. Since $p>2$, by \cite{KP} it follows that $(x_n)$ is either equivalent to the unit basis of $\ell_{p}$ or $\ell_2$. Suppose that $(x_n)$ is equivalent to the unit basis of $\ell_2$, then for arbitrary scalars $(a_n)$ we have
$$
\Big\|\sum_n a_n x_n\Big\|_p\approx\Big(\sum_n|a_n|^2\Big)^{\frac12}\approx\Big\|\sum_n a_n i_{p,2}Tx_n\Big\|_2\leq \Big\|\sum_n a_n Tx_n\Big\|_{p}.
$$
Thus, $T$ is an isomorphism on the subspace generated by $(x_n)$ which is a contradiction with $T$ being $\ell_2$-singular. Therefore, $(x_n)$ must be equivalent to the unit vector basis of $\ell_{p}$, but this would  imply that for every $k\in\mathbb N$
$$
k^{\frac12}\lesssim \Big\|\sum_{n=1}^k i_{p,2}Tx_n\Big\|_2\lesssim \Big\|\sum_{n=1}^k x_n\Big\|_{p}\lesssim k^{\frac1{p}}.
$$
This is impossible for large $k$ as $p>2$, so we conclude that $\liminf_n\|i_{p,2}Tx_n\|_2=0$ and $i_{p,2}T$ is compact, as claimed.
\end{proof}

Next result is based on a well-known interpolation construction from \cite{DFJP} (see also \cite[Section 5.2]{AB}).
\begin{lemma}\label{l:factorSS}
Let $1<p<2$ and $T\in\mathcal{S}(L_p)$. There exist a Banach lattice $X_T$ with the following properties:
\begin{enumerate}[\rm(i)]
\item $j\colon X_T\hookrightarrow L_p$ is bounded.
\item $\widetilde T:L_p\rightarrow X_T$ given by $\widetilde Tx=Tx$ is bounded.
\item The unit ball of $X_T$ is uniformly $p$-integrable in $L_p$.
\item $\widetilde T$ is strictly singular.
\item The composition $\widetilde T i_{2,p}:L_2\rightarrow X_T$ is compact.
\end{enumerate}
\end{lemma}
\begin{proof}
Let $W$ denote the solid convex hull of $T(B_{L_p})$.
Clearly, $W$ is convex, solid and uniformly $p$-integrable in $L_p$, by Lemma~\ref{l:ellpsingular}.
For each $n\in\mathbb N$, let $U_n=2^nW+2^{-n} B_{L_p}$, and denote
$$
\|x\|_n=\inf\{\lambda>0:x\in\lambda U_n\}.
$$
Let us define
$$
X_T=\Big\{x\in L_p:\|x\|_{X_T}=\Big(\sum_{n\in\mathbb N} \|x\|_n^2\Big)^{\frac12}<\infty\Big\}.
$$
By \cite[Theorem 5.41]{AB}, it follows that $X_T$ is a Banach lattice, that the operator
$\widetilde T\colon L_p\rightarrow X_T$ given by $\widetilde T(x)=T(x)$ is bounded, and the inclusion
$j\colon X_T\rightarrow L_p$ is also bounded (see also \cite[Theorem 5.37]{AB}). Thus, we have (i) and~(ii).

For the proof of (iii), let $\varepsilon>0$ and let $n\in\mathbb N$ such that $2^{-n+1}<\varepsilon\leq 2^{-n+2}$.
Since $W$ is uniformly $p$-integrable, by Lemma~\ref{l:uniformly p-integrable} there is $M_{\varepsilon/2}>0$ such that
$W\subseteq M_{\varepsilon/2} B_{L_\infty}+\frac{\varepsilon}{2} B_{L_p}$.
Now let $x\in X_T$ such that $\|x\|_{X_T}\leq1$. In particular, we have that $\|x\|_n\leq1$, or in other words,
\begin{multline*}
x\in 2^nW+2^{-n} B_{L_p}\subseteq 2^n M_{\varepsilon/2}B_{L_\infty}+2^{-n} B_{L_p}+\frac{\varepsilon}{2} B_{L_p}\\
  \subseteq \frac{4 M_{\varepsilon/2}}{\varepsilon}B_{L_\infty}+\varepsilon B_{L_p}.
\end{multline*}
Hence, $B_{X_T}\subseteq \frac{4 M_{\varepsilon/2}}{\varepsilon}B_{L_\infty}+\varepsilon B_{L_p}$,
and since this holds for every $\varepsilon>0$, it follows by Lemma \ref{l:uniformly p-integrable} that $B_{X_T}$ is uniformly $p$-integrable in~$L_p$.

Recall an operator $S\colon E\rightarrow F$ is strictly singular if and only if for every infinite-dimensional subspace $X\subseteq E$, there is a further infinite-dimensional subspace $Y\subseteq X$ such that the restriction $S|_{Y}$ is compact. Thus, (iv) follows from \cite[Theorem 5.40]{AB}.

Finally, for the proof of (v), note first that $T i_{2,p}\colon L_2\rightarrow L_p$ is compact. Indeed, since $T\in\mathcal S(L_p)$,
we have that $T^*\in\mathcal S(L_{p'})$ and by Lemma~\ref{l:ell2singular}, it follows that
$i_{p',2}T^*\colon L_{p'}\rightarrow L_2$ is compact. By duality, we have that
$T i_{2,p}\colon L_2\rightarrow L_p$ is compact and the conclusion follows from \cite[Theorem 5.40]{AB}.
\end{proof}

\begin{proof}[Proof of Theorem \ref{t:factorizationSS}]
By Lemma~\ref{l:factorSS} we have the factorization
$$
\xymatrix{L_p\ar_{\widetilde{T}}[dr]\ar[rr]^{T}&&L_p\\
&X_T\ar[ru]_{j}&}
$$
where $j$ maps the unit ball into a uniformly $p$-integrable set.

Moreover, the composition $\widetilde T i_{2,p}\colon L_2\rightarrow X_T$ is compact.
From this fact and duality, using \cite{J} it follows that $\widetilde T$ actually factors through $\ell_p$:
$$
\xymatrix{L_p\ar_{R}[dr]\ar[rr]^{\widetilde T}&&X_T\\
&\ell_p\ar[ru]_{S}&}
$$
Joining the two diagrams we get the result.
\end{proof}

\begin{remark}
For the sake of completeness, let us briefly recall the factorization construction given in~\cite{J}
for an operator $A\colon L_p\rightarrow L_p$ with $p>2$:
Let $(h_n)_{n\in\mathbb N}$ denote the Haar basis for $L_p$. For increasing sequences $(k_n)_{n\in\mathbb N}$,
$(M_n)_{n\in\mathbb N}$ in $\mathbb N$, take $H_n=[h_i]_{i=k_n}^{k_{n+1}-1}$
and $|f|_n=\max\{M_n\|f\|_{L_2},\|f\|_{L_p}\}$. It can be checked that $Y=(\oplus_{n\in\mathbb N} (H_n,|\cdot|_n))_{\ell_p}$ is isomorphic to~$\ell_p$.
Moreover, for $x\in L_p$, set $A_1(x)=(y_n)_{n\in\mathbb N}\in Y$, where $y_n\in H_n$ are such that
$A(x)=\sum_{n\in \mathbb N} y_n$, and for $(x_n)_{n\in\mathbb N}\in Y$ set $A_2((x_n)_{n\in\mathbb N})=\sum_{n\in \mathbb N} x_n$.
The proof given in~\cite{J} yields that if $i_{p,2}A$ is compact, then there exist increasing sequences
$(k_n)_{n\in\mathbb N}$, $(M_n)_{n\in\mathbb N}$ such that the corresponding $A_1$ and $A_2$ are   bounded, and clearly $A=A_2 A_1$.
\end{remark}
\begin{remark}
The Banach lattice $X_T$ constructed in the lemma (and thus in Theorem~\ref{t:factorizationSS}) can actually be taken to be a rearrangement invariant space. This can be done by taking $W$ closed under measure rearrangements of the underlying space $[0,1]$.
We refer the reader to \cite{LT2} for background on rearrangement invariant spaces. Similarly, for $q>p$ one can also
achieve that $L_q\subseteq X_T$ by enlarging $W$ in the above proof.
\end{remark}

Recall that an operator $T\colon E\rightarrow Y$ on a Banach lattice $E$ is called M-weakly compact if $\|Tx_n\|\rightarrow 0$ for every sequence of pairwise disjoint normalised vectors $(x_n)\subseteq E$. By duality, the following is an immediate consequence of Theorem~\ref{t:factorizationSS}.
\begin{corollary}\label{c:teop>2}
Let $p>2$ and $T\in\mathcal S(L_p)$. There are a rearrangement invariant space $X_T$ and operators
$T_1\colon L_p\rightarrow X_T$, $T_2\colon X_T\rightarrow \ell_p$ and $T_3\colon \ell_p\rightarrow L_p$ making
the following diagram commutative
$$
\xymatrix{L_p\ar_{T_1}[d]\ar[rr]^{T}&&L_p\\
X_T\ar[rr]^{T_2}&&\ell_p\ar_{T_3}[u] }
$$
with $T_1$ M-weakly compact.
\end{corollary}
\begin{proof}
If $T\in\mathcal S(L_p)$, then $T^*\in\mathcal S(L_{p'})$ \cite{W}. Now since $p'<2$, Theorem~\ref{t:factorizationSS} gives us the
factorization $T^*=j S R$. Let $T_1=j^*$, $T_2=S^*$ and $T_3=R^*$. By~\cite[Theorem 5.64]{AB}, and the fact that $j(B_{L_{p'}})$ is uniformly $p$-integrable, it follows that $T_1=j^*$ is M-weakly compact.
\end{proof}
\begin{theorem}\label{t:factorKellp}
Let $A,B\in \mathcal S(L_p)$. Then $L_AR_B$ factors through $\mathcal K(\ell_p)$.
\end{theorem}
\begin{proof}
By passing to adjoints we can assume without loss of generality that $p>2$. First, let us show that $A$ and $B$ have factorization diagrams through the same spaces. To this end, let $X_1$ be the subspace of $L_p$ consisting of functions supported on $[0,1/2]$ and $X_2$ those supported on $[1/2,1]$. Clearly, we have the (band) decomposition $L_p=X_1\oplus X_2$ and lattice isomorphisms $\varphi_i\colon X_i\rightarrow L_p$ for $i=1,2$.
Let $j_i\colon X_i\rightarrow L_p$ denote the inclusion operators, and $P_i\colon L_p\rightarrow X_i$ the corresponding (band) projections for $i=1,2$.
Let us consider the operator
$$
T=j_1\varphi_1^{-1}A\varphi_1P_1+j_2\varphi_2^{-1}B\varphi_2P_2.
$$
Clearly, $T\in\mathcal S (L_p)$, so Corollary~\ref{c:teop>2} yields the factorization
$$
\xymatrix{L_p\ar_{T_1}[d]\ar[rr]^{T}&&L_p\\
X_T\ar[rr]^{T_2}&&\ell_p\ar_{T_3}[u] }
$$
with $T_1$ M-weakly compact. It follows that we can factor $A$ and $B$ as follows
$$
\xymatrix{L_p\ar_{A_1}[d]\ar[rr]^{A}&&L_p&&L_p\ar_{B_1}[d]\ar[rr]^{B}&&L_p\\
X_T\ar[rr]^{A_2}&&\ell_{p}\ar_{A_3}[u] && X_T\ar[rr]^{B_2}&&\ell_{p}\ar_{B_3}[u] }
$$
where $A_1$ and $B_1$ are M-weakly compact.

Therefore, we can write $L_AR_B=(L_{A_3}R_{B_1})\circ (L_{A_2}R_{B_2})\circ (L_{A_1}R_{B_3})$. We claim that
$$
L_{A_1}R_{B_3}(\mathcal L(L_p))\subseteq \mathcal K(\ell_p,X_T).
$$
Indeed, assuming the contrary, let $T\in\mathcal L(L_p)$ and take a norm bounded sequence $(x_n)\subseteq \ell_p$ such that $(A_1TB_3x_n)$ has no convergent subsequence. Passing to a further subsequence we can assume that $(x_n)$ is weakly null and equivalent to the unit vector basis of $\ell_p$. By \cite{KP}, it follows that up to a further subsequence $(TB_3 x_n)\subseteq L_p$ is either equivalent to the unit vector basis of $\ell_2$ or $\ell_p$. In the former case, as $p>2$, it would follow that $\|TB_3x_n\|\rightarrow 0$ by Pitt's theorem. In the latter, we actually have that $\|TB_3x_n-y_n\|\rightarrow 0$ for a certain pairwise disjoint sequence $(y_n)\subseteq L_p$. Now, since $A_1$ is M-weakly compact, it follows that $\|A_1TB_3x_n\|\rightarrow0$. This is a contradiction, hence $L_{A_1}R_{B_3}(\mathcal L(L_p))\subseteq \mathcal K(\ell_p,X_T)$, as claimed.

In particular, we have the following factorization:
$$
\xymatrix{\mathcal L(L_p)\ar_{L_{A_1}R_{B_3}}[d]\ar[rr]^{L_{A} R_{B}}&&\mathcal K(L_p)\\
\mathcal K(\ell_p,X_T)\ar^{L_{A_2}R_{B_2}}[rr]\ar_{L_{A_2}}[rd]&&\mathcal K(X_T,\ell_p)\ar_{L_{A_3}R_{B_1}}[u]\\
&\mathcal K(\ell_p)\ar_{R_{B_2}}[ru]&}
$$

\vspace{-.4cm}
\end{proof}
\begin{remark}
Let us briefly recall part of the strategy of proof of  \cite[Theorem 2.9]{LST}: after some preliminary block-diagonalization argument, the authors show that if $A,B\in\mathcal S(L_p)$ ($p>2$) and $L_AR_B$ were not strictly singular, then it must be invertible on a subspace isomorphic to $\ell_s$, with $s=\frac{2p}{p-2}$ (see Claim 1 in the proof of  \cite[Theorem 2.9]{LST}); from there, some more work is necessary to reach a contradiction with the strict singularity of $B$. Alternatively, Theorem \ref{t:factorKellp} together with Lemma~\ref{l:p-Holub}, yield that if $A,B\in \mathcal S(L_p)$, and $L_AR_B$ is invertible in some subspace $X\subseteq \mathcal L(L_p)$, then $X$ should contain a subspace isomorphic to $\ell_p$, $\ell_{p'}$ or $c_0$ (although this last option is impossible because of the weak compactness of $L_AR_B$). Hence,  the case $X=\ell_s$ as above can also be ruled out by this approach.
\end{remark}
\begin{remark}
We do not know whether every $T\in\mathcal S(L_p)$ is approximately $(\ell_r,\ell_s)$-factorizable for some $r<s$. In fact, we do not know whether for an operator $T\in\mathcal S(L_p)$, and every $\varepsilon>0$, there exist $r(\varepsilon)<s(\varepsilon)$, two sequences of finite dimensional spaces $(X_n)_{n\in\mathbb N}$, $(Y_n)_{n\in\mathbb N}$ and operators
$$
\begin{array}{cl}
T_1^\varepsilon:&L_p\longrightarrow (\oplus X_n)_{\ell_{r(\varepsilon)}}\\
T_2^\varepsilon:&(\oplus X_n)_{\ell_{r(\varepsilon)}}\longrightarrow (\oplus Y_n)_{\ell_{s(\varepsilon)}}\\
T_3^\varepsilon:&(\oplus Y_n)_{\ell_{s(\varepsilon)}}\longrightarrow  L_p
\end{array}
$$
such that $\|T- T_3^\varepsilon T_2^\varepsilon T_1^\varepsilon\|\leq \varepsilon$. Keeping in mind the previous comment, by Remark \ref{r:sums} and the argument in the proof of Theorem \ref{t:factorKellp}, such a factorization would yield an alternative direct proof of \cite[Theorem 2.9]{LST}.
\end{remark}

In connection with approximate factorization the following property of strictly singular operators might be useful
(compare with \cite[Proposition~4.1]{HST}). For a measurable set $A\subseteq [0,1]$, let $P_A$ denote the projection onto the band of functions supported on $A$: $P_A x=\chi_A x$.

\begin{proposition}
Let $T\in\mathcal L(L_p)$ be $\ell_p$-singular. When $p<2$, for every $\varepsilon>0$, there is $\delta>0$ such that if $\mu(A)<\delta$, then $\|P_AT\|\leq \varepsilon$. Similarly, when $p\geq2$, for every $\varepsilon>0$, there is $\delta>0$ such that if $\mu(A)<\delta$, then $\|TP_A\|\leq \varepsilon$.
\end{proposition}
\begin{proof}
By duality it is enough to proof the first statement. By Lemma \ref{l:ellpsingular}, $T(B_{L_p})$ is a uniformly $p$-integrable set. Thus, for every $\varepsilon>0$ there is $M_\varepsilon>0$ such that $T(B_{L_p})\subseteq M_\varepsilon B_{L_\infty}+\frac{\varepsilon}{2} B_{L_p}$. Let $\delta=(\varepsilon/2M_\varepsilon)^p$. It follows that for any set with $\mu(A)<\delta$
$$
\|P_A T\|=\sup_{x\in B_{L_p}}\|P_A Tx\|_p\leq \sup_{x\in B_{L_p}} \|P_A M_\varepsilon\|_p+\frac{\varepsilon}{2}\leq \varepsilon.
$$
\end{proof}


\begin{thebibliography}{25}
\bibitem{AK}
F. Albiac, N. J. Kalton, Topics in Banach space theory. Graduate Texts in Mathematics, 233. Springer, New York, 2006.

\bibitem{AB}
C. D. Aliprantis and O. Burkinshaw, Positive operators. Reprint of the 1985 original. Springer, Dordrecht, 2006.

\bibitem{CuMM}
R. E. Curto and M. Mathieu (eds.),
Proc. 3rd Int. Workshop on Elementary Operators and their Applications, Belfast, 2009;
Operator Theory: Advances and Applications, 212, Birkh\"auser Verlag, Basel, 2011.

\bibitem{DFJP}
W. J. Davis, T. Figiel, W. B. Johnson and A. Pe{\l}czy\'{n}ski, Factoring weakly compact operators. J. Functional Analysis 17 (1974), 311--327.

\bibitem{FHST}
J. Flores, F. L. Hern\'andez, E. M. Semenov, P. Tradacete, Strictly singular and power-compact operators on Banach lattices. Israel J. Math. 188 (2012), 323--352.

\bibitem{HST}
F. L. Hern\'andez, E. M. Semenov, P. Tradacete, Strictly singular operators on $L_p$ spaces and interpolation. Proc. Amer. Math. Soc. 138 (2010), no. 2, 675--686.

\bibitem{Holub}
J. R. Holub, On subspaces of separable norm ideals. Bull. Amer. Math. Soc. 79 (1973), 446--448.

\bibitem{J}
W. B. Johnson, Operators into $L_p$ which factor through $\ell_p$. J. London Math. Soc. (2) 14 (1976), no. 2, 333--339.

\bibitem{JS}
W. B. Johnson, G. Schechtman, Multiplication operators on $L(L_p)$ and $\ell_p$-strictly singular operators. J. Eur. Math. Soc. (JEMS) 10 (2008), no. 4, 1105--1119.

\bibitem{KP}
M. I. Kadec, A. Pelczynski, Bases, lacunary sequences and complemented subspaces in the spaces $L_p$. Studia Math. 21 (1961/62), 161--176.

\bibitem{K}
N. J. Kalton, Spaces of compact operators. Math. Ann. 208 (1974), 267--278.

\bibitem{Kato}
T. Kato, Perturbation theory for nullity deficiency and order quantities of linear operators. J. Analyse. Math. 6 (1958), 273--322.

\bibitem{LT}
J. Lindenstrauss, L. Tzafriri. Classical Banach Spaces. I. Sequence spaces.
Ergebnisse der Mathematik und ihrer Grenzgebiete, Vol.~92. Springer Verlag, Heidelberg, 1977.

\bibitem{LT2}
J. Lindenstrauss, L. Tzafriri. Classical Banach Spaces. II. Function spaces.
Ergebnisse der Mathematik und ihrer Grenzgebiete, Vol.~97. Springer Verlag, Heidelberg, 1979.

\bibitem{LST}
M. Lindstr\"om, E. Saksman, H.O. Tylli, Strictly singular and cosingular multiplications. Canad. J. Math. 57 (2005), no. 6, 1249--1278.

\bibitem{LS}
M. Lindstr\"om and G. Schl\"uchtermann, Composition of operator ideals. Math. Scand. 84(1999), 284--296.

\bibitem{MM92}
M. Mathieu (ed.), Proc. Int. Workshop on Elementary Operators and Applications, Blaubeuren, 1991;
World Scientific, Singapore, 1992.

\bibitem{M}
V. D. Milman. Operators of class $C_0$ and $C^*_0$. Teor. Funkci\u{\i} Funkcional. Anal. i Prilo\v{z}en. No. 10 (1970), 15--26.

\bibitem{OS}
T. Oikhberg and E. Spinu, Ideals of operators on $C^*$-algebras and their preduals. Bull. Lond. Math. Soc. 47 (2015), no. 1, 156--170.

\bibitem{Pelczynski}
A. Pe{\l}czy\'{n}ski, On strictly singular and strictly cosingular operators. I. Strictly singular and strictly cosingular operators in $C(S)$-spaces. Bull. Acad. Polon. Sci. S\'er. Sci. Math. Astronom. Phys. 13 1965 31--36.

\bibitem{ST}
E. Saksman and H.-O. Tylli, Weak compactness of multiplication operators on spaces of bounded linear operators. Math. Scand. 70 (1992), 91--111.

\bibitem{ST:survey}
E. Saksman, H.-O. Tylli, Multiplications and elementary operators in the Banach space setting. Methods in Banach space theory, 253--292, London Math. Soc. Lecture Note Ser., 337, Cambridge Univ. Press, Cambridge, 2006.

\bibitem{SSTJT}
B. Sari, T. Schlumprecht, N. Tomczak-Jaegermann, V. G.Troitsky, On norm closed ideals in $L(\ell_p,\ell_q)$. Studia Math. 179 (2007), no. 3, 239--262.

\bibitem{V}
K. Vala, On compact sets of compact operators. Ann. Acad. Sci. Fenn. Ser. A I Math. 351(1964).

\bibitem{W}
L. Weis. On perturbations of Fredholm operators in $L_p(\mu)$-spaces. Proc. Amer. Math. Soc. 67 (1977), 287--292.

\bibitem{Woj}
P. Wojtaszczyk, Banach spaces for analysts. Cambridge Studies in Advanced Mathematics, 25. Cambridge University Press, Cambridge, 1991.

\end{thebibliography}
\end{document}